\documentclass[a4paper,10pt]{amsart}
\usepackage[centertags]{amsmath}
\usepackage{amsfonts}
\usepackage{amssymb}
\usepackage{amsthm}
\usepackage{newlfont}
\usepackage{graphicx}
\usepackage{epstopdf}
\usepackage{blkarray}
\usepackage{multirow}
\usepackage{cite}
\usepackage{bm}
\newcommand{\hookuparrow}{\mathrel{\rotatebox[origin=c]{90}{$\hookrightarrow$}}}

\newcommand{\Z}{\mathbb{Z}}

\newcommand{\bq }{\begin{equation}}
\newcommand{\eq }{\end{equation}}
\theoremstyle{plain}
\newtheorem{thm}{Theorem}
\newtheorem{lem}[thm]{Lemma}
\newtheorem{prop}[thm]{Proposition}

\newtheorem{cor}[thm]{Corollary}

\newtheorem{rem}[thm]{Remark}
\theoremstyle{definition}

\theoremstyle{example}

\title{Automorphism groups of non-singular plane curves of degree 5}

\author[E. Badr] {Eslam Badr}
\address{$\bullet$\,\,Eslam Badr}
\address{Departament Matem\`atiques, Edif. C, Universitat Aut\`onoma de Barcelona\\
08193 Bellaterra, Catalonia}
\email{eslam@mat.uab.cat}
\address{Department of Mathematics,
Faculty of Science, Cairo University, Giza-Egypt}
\email{eslam@sci.cu.edu.eg}

\author[F. Bars] {Francesc Bars}
\address{$\bullet$\,\,Francesc Bars}
\address{Departament Matem\`atiques, Edif. C, Universitat Aut\`onoma de Barcelona\\
08193 Bellaterra, Catalonia} \email{francesc@mat.uab.cat}
\thanks{E. Badr and F. Bars are supported by MTM2013-40680-P}



\begin{document}

\maketitle
\begin{abstract}
Let $M_g$ be the moduli space of smooth, genus $g$ curves over an
algebraically closed field $K$ of zero characteristic. Denote by
${M_g(G)}$ the subset of $M_g$ of curves $\delta$ such that $G$ (as
a finite non-trivial group) is isomorphic to a subgroup of
$Aut(\delta)$, and let $\widetilde{M_g(G)}$ be the subset of curves
$\delta$ such that $G\cong Aut(\delta)$, where $Aut(\delta)$ is the
full automorphism group of $\delta$. Now, for an integer $d\geq 4$,
let $M_g^{Pl}$ be the subset of $M_g$ representing smooth, genus
$g$, plane curves of degree $d$ (in this case, $g=(d-1)(d-2)/2$) and
consider the sets $M_g^{Pl}(G):=M_g^{Pl}\cap M_g(G)$ and
$\widetilde{M_g^{Pl}(G)}:=\widetilde{M_g(G)}\cap M_g^{Pl}$.

Henn in \cite{He} and Komiya-Kuribayashi in \cite{KuKo2}, listed the
groups $G$ for which $\widetilde{M_3^{Pl}(G)}$ is non-empty. In this
paper, we determine the loci $\widetilde{M_6^{Pl}(G)}$,
corresponding to non-singular degree $5$ projective plane curves,
which are non-empty. Also, we present the analogy of Henn's results
for quartic curves concerning non-singular plane model equations
associated to these loci (see Table 2 for more details). Similar
arguments can be applied to deal with higher degrees.
\end{abstract}

\section{Introduction}

It is classically known from Hurwitz \cite{Hur} that, given any
non-trivial finite group $G$, one can construct a Riemann surface
$X$ whose automorphism group $Aut(X)$ is isomorphic to $G$.

A natural question is to list the groups such that the associated
Riemann surface will have a non-singular plane model. Harui in
\cite{Harui} determined the list of the finite groups $G$ that could
appear in such case. However, for a complete answer to the problem,
it remains to introduce the exact list of such groups which might
appear for a fixed degree and conversely, for an arbitrary but fixed
group in the list, one need to determine the degrees for which such a
group occurs. Therefore, there are the following two open problems:
\begin{enumerate}
\item Fixing a group
$G$, for which degrees $d\geq 4$, $\widetilde{M_g^{Pl}(G)}$ is a non-empty set? For example, by the
work of Crass in \cite{Crass}, we know that
$\widetilde{M_g^{Pl}(A_6)}$ is non-empty exactly for $g=10$,
$g=55$ and $g=406$, where $A_6$ is the alternating group of 6 letters. 

\item Once the degree $d$ is fixed, determine the groups $G$
(up to isomorphism) where $\widetilde{M_g^{Pl}(G)}$ is non-empty.

\end{enumerate}

This note is concerned with the second question. Henn in \cite{He}
and Komiya-Kuribayashi in \cite{KuKo2} solved the question for
$d=4$.
\par Recall that any $\delta\in M_g^{Pl}(G)$ corresponds to a
set of non-singular plane models $C_{\delta}$ in $\mathbb{P}^2(K)$
such that any two of them are related through a change of variables
$P\in PGL_3(K)$ (where $PGL_n(K)$ is the classical projective linear
group of $n\times n$ invertible matrices over $K$), and their
automorphism groups are conjugate. By $C$ we mean a plane
non-singular model associated to $\delta$. Observe that $Aut(C)$ is
a subgroup of $PGL_3(K)$ which is isomorphic to $G$ by an injective
representation $\rho:G\hookrightarrow PGL_3(K)$, that is
$Aut(C)=\rho(G)$ for some $\rho$.
\par We denote by $\rho(M_g^{Pl}(G))$ the set of all elements $\delta\in
M_g^{Pl}(G)$ such that $G$ acts on a plane model associated to
$\delta$ as $P^{-1}\rho(G)P$ for some $P$. This gives us the
following union decomposition:
$$M_g^{Pl}(G)=\cup_{[\rho]\in A}\rho (M_g^{Pl}(G))$$
where  $A:=\{\rho\,\,|\,\,\rho:G\hookrightarrow PGL_3(K)\}/\sim$
such that $\rho_a\sim\rho_b$ if and only if
$\rho_a(G)=P^{-1}\rho_b(G)P$ for some $P\in PGL_3(K)$. A similar
decomposition (which is now disjoint) follows for
$\widetilde{M_g^{Pl}(G)}$.

Henn in \cite{He} determined the $[\rho]'s$ and $G$ such that
$\widetilde{\rho(M_3^{Pl}(G))}$ is non-trivial, and associated to
such locus (once $\rho$ and $G$ are fixed) a certain projective
plane equation which depends on some parameters with respect to some
algebraic restrictions. More concretely, he
obtained a plane non-singular model of any element of the locus through a
certain specialization of the values of the parameters and vice
versa.

\par In this paper, we obtain the analogy of the previous Henn's results for the loci
$\widetilde{\rho(M_6^{Pl}(G))}$, see Table 2 for a compact form of
the analogy, which is also the main result of the paper.

\noindent First, we classified in \cite{BaBacyc}, for an arbitrary
but a fixed degree $d$, the $\rho$'s and the cyclic groups $\Z/m\Z$
of order $m$ such that $\rho(M_g^{Pl}(\Z/m\Z))$ is not empty. In
particular, $m$ should divide one of the integers $$d,d-1,
d(d-1),(d-1)^2,d(d-2)\,\,\textit{or}\,\,d^2-3d+3.$$ Furthermore, we
characterized the locus $\widetilde{M_6^{Pl}(G)}$ whenever $G$ has
an element of order $m$ with $m$ large enough. By large enough, we
mean to be one of the following integers:
$d(d-1),(d-1)^2,d(d-2),d^2-3d+3,\ell d$ ($\ell\geq 3$) or
$\ell(d-1)$ ($\ell\geq 2$). Lastly, it remains to treat case by case
the groups $G$ that appeared in Harui's list \cite{Harui} in order
to investigate which of them must be left when the locus
$\widetilde{M_6^{Pl}(G)}$ is non-trivial and $G$ has no elements of
(large enough) order $m$.

\noindent We thank a referee for his or her comments and suggestions
that improved the paper in its present form.

\section{Cyclic subgroups for degree 5 non-singular plane curve}
Consider $\delta\in M_{6}^{Pl}$ such that the group $G\cong
Aut(\delta)$ is non-trivial. Let $C:\,\,F(X;Y;Z)=0$ in
$\mathbb{P}^2(K)$ be a non-singular plane model of degree 5 over an
algebraic closed field $K$ of characteristic zero, where
$Aut(C)=\rho(G)\leq PGL_3(K)$ for some $\rho:G\hookrightarrow
PGL_3(K)$ (any other model $C$ of $\delta$ is given by
$C_{P}:\,F(P(X;Y;Z))=0$ with $Aut(C_{P})=P^{-1}Aut(C)P$ for some
$P\in PGL_3(K)$, and we say that $C_{P}$ is $K$-equivalent or
$K$-isomorphic to $C$).
Assume that $\sigma\in Aut(C)$ is an element of order $m$ hence by a
change of variables in $\mathbb{P}^2$ (in particular, changing the
plane model to a $K$-equivalent one associated to $\delta$), we can
consider $\sigma$ as the automorphism $(x:y:z)\mapsto (x:\xi_m^a
y:\xi_m^b z)$ where $\xi_m$ is a primitive $m$-th root of unity in
$K$, and $a,b$ are integers such that $0\leq a<b\leq m-1$. Moreover,
if $ab\neq 0$ then $m$ and $gcd(a,b)$ are coprime (we can reduce to
$gcd(a,b)=1$) and if $a=0$ then $gcd(b,m)=1$. Also, such an
automorphism is identified with type $m,(a,b)$.
\par Then, by a change of variables, we may have one of the following
situations (see \cite{BaBacyc} for more details, in which we follow
the same line of argument as Dolgachev for degree $4$ in \cite{Dol},
but for a general degree $d\geq4$).

\begin{center}
\begin{table}[!th]
  \renewcommand{\arraystretch}{1.3}
  \caption{Quintics\,\,\,}\label{table:Cyclic Auto.}
  \centering
\begin{tabular}{|c|c|}
  \hline
  Type: $m, (a,\,b)$ & $F(X;Y;Z)$ \\\hline\hline
  $20,(4,5)$& $X^5+Y^5+XZ^4$ \\\hline
  $16,(1,12)$& $X^5+Y^4Z+XZ^4$\\\hline
  $15,(1,11)$& $X^5+Y^4Z+YZ^4$    \\\hline
$13,(1,10)$& $X^4Y+Y^4Z+Z^4X$    \\\hline
 $10,(2,5)$& $X^5+Y^5+\alpha XZ^4+\beta_{2,0}X^3Z^2$ \\\hline
  $8,(1,4)$& $X^5+Y^4Z+\alpha XZ^4+\beta_{2,0}X^3Z^2$ \\\hline
 $5,(1,2)$& $X^5+Y^5+Z^5+\beta_{3,1}X^2YZ^2+\beta_{4,3}XY^3Z$    \\\hline
  $5,(0,1)$& $Z^5+L_{5,Z}$    \\\hline

\end{tabular}
\end{table}
\end{center}

\begin{center}
\begin{tabular}{|c|c|}
\hline
  $4,(1,3)$& $X^5+X\big(Z^4+\alpha Y^4+\beta_{4,2}Y^2Z^2\big)+\beta_{2,1}X^3YZ$\\\hline
 $4,(1,2)$& $X^5+X\big(Z^4+\alpha Y^4\big)+\beta_{2,0}X^3Z^2+\beta_{3,2}X^2Y^2Z+\beta_{5,2}Y^2Z^3$\\\hline
 $4,(0,1)$& $Z^4L_{1,Z}+L_{5,Z}$    \\\hline
$3,(1,2)$& $X^5+Y^4Z+\alpha
YZ^4+\beta_{2,1}X^3YZ+X^2\big(\beta_{3,0}Z^3+\beta_{3,3}Y^3\big)+\beta_{4,2}XY^2Z^2$
\\\hline
 $2,(0,1)$& $Z^4L_{1,Z}+Z^2L_{3,Z}+L_{5,Z}$    \\\hline

  \end{tabular}
\end{center}

%
%
%
Here $L_{i,*}$ means a homogenous polynomial of degree $i$ in the
variables $\{X,Y,Z\}$ such that the variable $*$ does not appear.
Also, $\alpha, \beta_{i,j}\in K$ and $\alpha$ is always non-zero and
it can be transformed by a diagonal change of variables $P$ to 1.
\begin{rem} It is to be noted that the above table lists all the possible situations for which $\rho(M_6(\Z/m\Z))$ is not empty where
$P\rho(\Z/m\Z)P^{-1}=<(diag(1,\xi_m^a,\xi_m^b)>$ for some $P\in
PGL_3(K)$ and $\rho(M_6(\Z/m\Z))$ corresponds to Type $m,(a,b)$.
\end{rem}

\section{General properties of the full automorphism group}
Before a detailed study of the automorphism groups for degree 5, we
recall the following results concerning $Aut(\delta)$ for $\delta\in
M_g^{Pl}$, which will be useful throughout this paper. In some cases
we will use the notation of the GAP library for finite small groups
to indicate some of them.
\par Because linear systems $g^2_d$ are
unique (up to multiplication by $P\in PGL_3(K)$ in $\mathbb{P}^2(K)$
\cite[Lemma 11.28]{Book}), we always take $C$ as a plane non-singular
model of $\delta$, which is given by a projective plane equation
$F(X;Y;Z)=0$ and $Aut(C)$ as a finite subgroup of $PGL_3(K)$ that
fixes the equation $F$ and is isomorphic to $Aut(\delta)$. Any other
plane model of $\delta$ is given by $C_{P}:\,F(P(X;Y;Z))=0$ with
$Aut(C_{P})=P^{-1}Aut(C)P$ for some $P\in PGL_3(K)$ and $C_{P}$ is
$K$-equivalent or $K$-isomorphic to $C$. By an abuse of notation, we
also denote a non-singular projective plane curve of degree $d$ by
$C$. Therefore, $Aut(C)$ satisfies one of the following situations
(see Mitchell \cite{Mit} for more details):
\begin{enumerate}
\item fixes a point $Q$ and a line $L$ with $Q\notin L$ in
$PGL_3(K)$,
\item fixes a triangle (i.e. a set of three non-concurrent lines),
\item $Aut(C)$ is conjugate to a representation inside $PGL_3(K)$ of
one of the finite primitive groups namely, the Klein group
$PSL(2,7)$, the icosahedral group $A_5$, the alternating group
$A_6$, the Hessian group $Hess_{216}$ or to one of its subgroups $Hess_{72}$ or $Hess_{36}$.
\end{enumerate}

It is classically known that if a subgroup $H$ of automorphisms of a
non-singular plane curve $C$ fixes a point on $C$ then $H$ is cyclic
\cite[Lemma 11.44]{Book}, and recently Harui in \cite[\S2]{Harui}
provided the lacked result in the literature on the type of groups
that could appear for non-singular plane curves. Before introducing
Harui's statement, we need to define the terminology of being a
descendent of a plane curve: For a non-zero monomial $cX^iY^jZ^k$
with $c\in K\setminus\{0\}$, we define its exponent as
$max\{i,j,k\}$. For a homogenous polynomial $F$, the core of $F$ is
defined to be the sum of all terms of $F$ with the greatest
exponent. Let $C_0$ be a smooth plane curve, a pair $(C,H)$ with
$H\leq Aut(C)$ is said to be a descendant of $C_0$ if $C$ is defined
by a homogenous polynomial whose core is a defining polynomial of
$C_0$ and $H$ acts on $C_0$ under a suitable change of the
coordinate system (i.e. $H$ is conjugate to a subgroup of
$Aut(C_0)$).


There is a natural map $\varrho:PBD(2,1)\rightarrow PGL_2(K)$ given
by $[A']\mapsto [A'']$, where $[M]$ denotes the equivalence class of
the matrix $M$ in the projective linear group, where $A'$ is the
matrix given by a diagonal block matrices by entries $A''$ and
$\alpha\in K^*$.

\begin{thm}[Harui] \label{teoHarui} If
$H\preceq\,\,Aut(C)$, where $C$ is a non-singular plane curve of
degree $d\geq4$ then $H$ satisfies one of the following.
\begin{enumerate}
  \item $H$ fixes a point on $C$ and then it is cyclic.
  \item $H$ fixes a point not lying on $C$ and it satisfies a short exact sequence of the form
  $$1\rightarrow N\rightarrow H\rightarrow G'\rightarrow 1,$$
where $N$ is a cyclic group of order dividing $d$ and $G'$ (which is a
subgroup of $PGL_2(K)$) is conjugate to a cyclic group $\Z/m\Z$ of
order $m$ with $m\leq d-1$, a Dihedral group $D_{2m}$ of order $2m$
with $|N|=1$ or $m|(d-2)$, the alternating groups $A_4$, $A_5$ or
the symmetry group $S_4$. Moreover, we have the following
commutative diagram with exact rows and vertical injective
morphisms:
$$\begin{array}{ccccccccc}
1&\rightarrow&K^*&\rightarrow&
PBD(2,1)&\rightarrow^{\varrho}&PGL_2(K)&\rightarrow&1\\
&&\hookuparrow&&\hookuparrow&&\hookuparrow\\
1&\rightarrow& N&\rightarrow&H&\rightarrow&G'&\rightarrow&1 .\\
\end{array}$$

\item $H$ is conjugate to a subgroup of $Aut(F_d)$ where $F_d$ is the Fermat curve $X^d+Y^d+Z^d$.
In particular, $|H|\,|\,6d^2$ and $(C,H)$ is a descendant of $F_d$.
\item $H$ is conjugate to a subgroup of $Aut(K_d)$ where $K_d$ is the Klein curve curve $XY^{d-1}+YZ^{d-1}+ZX^{d-1}$
hence $|H|\,|\,3(d^2-3d+3)$ and $(C,H)$ is a descendant of $K_d$.
\item $H$ is conjugate to a finite primitive subgroup of $PGL_3(K)$ which are mentioned above.

\end{enumerate}
\end{thm}

We have also the following statement \cite[Theorem 2.3]{Harui}.
\begin{thm} Given a non-singular plane curve $C$ of degree $d\neq4,6$, then $|Aut(C)|\leq 6 d^2.$
\end{thm}
In particular, for $d=5$ we conclude:
\begin{cor} Given $C$, a non-singular plane curve of degree 5, then
$Aut(C)$ is not conjugate to the Hessian group $Hess_{216}$, the
Klein group $PSL(2,7)$ or the alternating group $A_6$.
\end{cor}

Moreover, we proved in \cite{BaBacyc} the following two results for
cyclic subgroups inside $Aut(C)$ (see \cite[Corollary 33]{BaBacyc}
and \cite[\S4]{BaBacyc} respectively).
\begin{prop} Let $C$ be a non-singular plane curve of degree $d$, and
let $\sigma\in Aut(C)$ be of order $m$. Then $m$ divides one of the
following integers: $d-1,\,\,d,\,\,(d-1)^2,\,\,d(d-2),\,\,d(d-1)$ or
$d^2-3d+3$.
\end{prop}

\begin{thm}\label{BB} Let $C$ be a non-singular plane curve of degree
$d$ with $\sigma\in Aut(C)$. Then,
\begin{enumerate}
\item if $\sigma$ has order $d(d-1)$ then $Aut(C)=<\sigma>$, and $C$
is $K$-isomorphic to $X^d+Y^d+XZ^{d-1}=0$.
\item if $\sigma$ has order $(d-1)^2$ then $Aut(C)=<\sigma>$, and
$C$ is $K$-isomorphic to $X^d+Y^{d-1}Z+XZ^{d-1}=0$.
\item if $\sigma$ has order $d(d-2)$ then $C$ is $K$-isomorphic to
$X^d+Y^{d-1}Z+YZ^{d-1}=0$, and for $d\neq 4,6$ we have
$Aut(C)=<\sigma,\tau|\tau^2=\sigma^{d(d-2)}=1,\, and\,
\tau\sigma\tau=\sigma^{-(d-1)}>.$
\item if $\sigma$ has order $d^2-3d+3$ then $C$ is $K$-isomorphic to the Klein
curve $K_d$, and for $d\geq 5$ we have
$$Aut(C)=<\sigma,\tau|\sigma^{d^2-3d+3}=\tau^3=1\, and\,
\sigma\tau=\tau\sigma^{-(d-1)}>.$$
\item if $\sigma$ has order $\ell (d-1)$ with $\ell\geq 2$ then
$Aut(C)$ is cyclic of order $\ell'(d-1)$ with $\ell|\ell'$. If
$\ell=1$, the same conclusion holds if $\sigma$ is a homology (i.e. $P\sigma P^{-1}=diag(1,\xi_m^a,\xi_m^b)$ such
that exactly one of $a$ or $b$ is zero for some $P\in PGL_3(K)$).
\item if $\sigma$ has order $\ell d$ with $\ell\geq 3$ then $Aut(C)$
fixes a line and a point off that line, and $Aut(C)$ is an exterior
group as in Theorem \ref{teoHarui} (2) with $N$ of order $d$. When
$\ell=2,\,C$ may be a decendent of the Fermat curve or $Aut(C)$ is
an exterior group as in Theorem \ref{teoHarui} (2) where $|N|=d$.
\end{enumerate}
\end{thm}
Now, assume, as usual, that $C$ is a non-singular plane curve of
degree $d=5$ with $\sigma\in Aut(C)$ of order $m$ that acts on
$F(X;Y;Z)=0$ as $(x,y,z)\mapsto (x,\xi_m^a y,\xi_m^bz)$ such that
$m$ is the maximal order in $Aut(C)$. Recall also that we can take
$\alpha=1$ by a convenient change of variables $P$.

The following result determines the full automorphism group of a
quintic curve $C$ which admits automorphisms of large orders.

\begin{cor} For non-singular plane curves of degree $5$
over an algebraic closed field $K$ of zero characteristic we have:
\begin{enumerate}
\item The cyclic group $\Z/20\Z$ appears as $Aut(C)$ inside $PGL_3(K)$, and is generated by
the transformation $(x,y,z)\mapsto (x,\xi_{20}^4 y,\xi_{20}^5 z)$ up
to conjugation by $P\in PGL_3(K)$. Moreover, $C$ is $K$-isomorphic
(through $P$) to the plane non-singular curve $X^5+Y^5+ XZ^4=0$. In
particular, $\widetilde{\rho(M_6^{Pl}(\Z/20\Z))}$ is an irreducible
locus with one element, where
$\rho(\Z/20\Z)=<diag(1,\xi_{20}^4,\xi_{20}^5)>$.
\item The cyclic group $\Z/{16}\Z$ appears as $Aut(C)$ inside
$PGL_3(K)$, and is generated by the transformation $(x,y,z)\mapsto
(x,\xi_{16} y,\xi_{16}^{12} z)$ up to conjugation by $P\in
PGL_3(K)$. Furthermore, $C$ is $K$-isomorphic (through $P$) to the
plane non-singular curve $X^5+Y^4 Z+ XZ^4=0$. In particular,
$\widetilde{\rho(M_6^{Pl}(\Z/16\Z))}$ is an irreducible locus with
one element, where $\rho(\Z/16\Z)=<diag(1,\xi_{16},\xi_{16}^{12})>$.
\item The group $SmallGroup(30,1)\,\,\cong\,\,<\sigma,\tau|\tau^2=\sigma^{15}=
(\tau\sigma)^2\sigma^{3}=1>$ of order 30 appears as $Aut(C)$ inside
$PGL_3(K)$, where $\sigma$ and $\tau$ are given, up to conjugation
through $P\in PGL_3(K)$, by $\sigma:(x,y,z)\mapsto (x,\xi_{15}
y,\xi_{15}^{11} z)$ and $\tau:(x,y,z)\mapsto (x,z,y)$. Moreover, $C$
is $K$-isomorphic (through $P$) to the curve $X^5+Y^4 Z+ YZ^4=0$. In
particular, $\widetilde{\rho(M_6^{Pl}(SmallGroup(30,1)))}$ is an
irreducible locus with one element and $\rho$ is given by
$\sigma,\tau$.
\item The group $SmallGroup(39,1)\,\,\cong\,\,<\tau,\sigma|\sigma^{13}=\tau^3=1,\,
\sigma\tau=\tau\sigma^{3}>$ of order 39 appears as $Aut(C)$ inside
$PGL_3(K)$ which is given by $\sigma:(x,y,z)\mapsto (x,\xi_{13}
y,\xi_{13}^{10} z)$ and $\tau:(x,y,z)\mapsto (y,z,x)$ up to
conjugation by $P\in PGL_3(K)$. Also, $C$ is $K$-isomorphic (through
$P$) to the curve $K_5:X^4 Y+Y^4 Z+ Z^4 X=0$. In particular,
$\widetilde{\rho(M_6^{Pl}(SmallGroup(39,1)))}$ is an irreducible
locus with one element, where $\rho$ is determined by $\sigma,\tau$.
\item The cyclic group $\Z/8\Z$ appears as $Aut(C)$ inside $PGL_3(K)$ that is generated by
the transformation $(x,y,z)\mapsto (x,\xi_{8} y,\xi_{8}^4 z)$ up to
conjugation by $P\in PGL_3(K)$, and $C$ is $K$-isomorphic (through
$P$) to the plane non-singular curve $X^5+Y^4Z+
XZ^4+\beta_{2,0}X^3Z^2$, with $\beta_{2,0}\neq 0,\pm 2$. The locus
$\widetilde{\rho(M_6^{Pl}(\Z/8\Z))}$ has dimension one where
$\rho(\Z/8\Z)=<diag(1,\xi_{8},\xi_{8}^4)>$.
\end{enumerate}
\end{cor}
\begin{proof} Except the last statement on the loci, the proof is a direct consequence of Theorem
\ref{BB}, because one could apply the result for $d=5$ using the
table in \S2 when the curve $C$ has a cyclic automorphism of order:
$d(d-1)$, $(d-1)^2$, $d(d-2)$, $d^2-3d+3$ and $\ell (d-1)$ with
$\ell=2$ respectively. It remains, for the last case, to observe
that if $Aut(C)$ is bigger then it is always cyclic and should be
the group of order $16$. Therefore, $\beta_{2,0}\neq 0$ is the only
restriction to impose so that the curve has automorphism group
exactly $\Z/8\Z$ (Here, we note that $\beta_{2,0}\neq\pm 2$ to
ensure non-singularity. Also $\alpha$ is converted to 1 through a
diagonal change of the variables). Lastly, we refer to
\cite{BaBacyc1} for the proof of the fact on the dimension and the
irreducibility over $\mathbb{C}$ of the locus
$\widetilde{\rho(M_6^{Pl}(\Z/8\Z))}$.
\end{proof}

\section{Determination of the automorphism group with small cyclic
subgroups}

In this section, following the abuse of notation of the previous section
concerning models and curves, we study $Aut(C)$ for non-singular
plane curves $C$ of degree $d=5$ that appear in the table of \S2
such that the maximal order for any element inside the automorphism
group is $2 d$ or $\leq d$.

Also, we denote by $C_n$ the cyclic group $\Z/n\Z$ to
emphasize the multiplication notation as a subgroup inside
$PGL_3(K)$.

\begin{prop} Suppose that $C$ is a non-singular plane curve of degree 5 with $\sigma\in Aut(C)$
of order $10$ as an automorphism of maximal order. Then, we reduce
after conjugation by certain $P\in PGL_3(K)$ that $\sigma$ acts on
$C:X^5+Y^5+XZ^4+\beta_{2,0}X^3Z^2=0$ such that $\beta_{2,0}\neq 0$
as $\sigma:(x,y,z)\mapsto (x,\xi_{10}^2 y,\xi_{10}^5 z)$, and one of
the following situations happens:
\begin{enumerate}
\item If $\beta_{2,0}^2=20$ then $C$ is $K$-equivalent to the Fermat quintic $F_5:\,X^5+Y^5+Z^5=0$ and $Aut(C)$ is isomorphic to SmallGroup$(150,5)$.
\item If $\beta_{2,0}^2\neq 20$ then $Aut(C)$ is cyclic of order $10$.
Moreover, $C$ is a descendant of the Fermat curve of the form
$C_{P}:\,\,X^5+Y^5+Z^5+u\left(\xi_{10}^{6}Y^4Z+YZ^4\right)+u'\left(\xi_{10}^{2}Y^3Z^2+Y^2Z^3\right)=0,$
with $(u,u')\neq(0,0)$.
\end{enumerate}
\end{prop}


\begin{proof}
Because the maximal order is $10$ then, by the results of the
previous section, we reduce $C$ to be, up to $K$-isomorphism, of the
form $X^5+Y^5+\alpha XZ^4+\beta_{2,0}X^3Z^2=0$ with
$\alpha\beta_{2,0}\neq 0$, and by a diagonal change of variables we
always can take $\alpha=1$. This curve admits a homology $\sigma^2$
of order $5>3$ therefore, by Theorem \ref{BB}, $Aut(C)$ fixes a line
and a point off that line or $C$ is a descendent of Fermat curve.
Moreover, the center $[0:1:0]$ of this homology is an outer Galois
point (by Lemma 3.7 \cite{Harui}) and if $C$ is not $K$-isomorphic
to the Fermat curve $F_5: X^5+Y^5+Z^5$ then it is unique (Theorem 4'
\cite{Yoshihara}). Hence it should be fixed by $Aut(C)$.
\par Assume first that $Aut(C)$ fixes a line and a point off that line and $C$ is not a descendent of the Fermat quintic.
Therefore, $Aut(C)$ satisfies a short exact sequence $1\rightarrow
C_5\rightarrow Aut(C)\rightarrow G'\rightarrow 1$, where $C_5$ is
generated by $\sigma^2=[X;\zeta_5^2 Y;Z]$. In particular, $G'$
contains an element of order $2$ which is obtained by the image of
$\sigma$ under the restriction of the natural map $\varrho$ from
$PBD(2,1)$ to $PGL_2(K)$. Consequently $G'$ is conjugate to $C_2,
C_4, S_3, A_4, S_4$ or $A_5$. We claim that $G'$ is conjugate to
$C_2$ (in particular $Aut(C)$ is cyclic of order $10$).
\par
\par Since there are no groups of order $30$ (respectively, $60$) which contain elements of order $10$ and no higher orders then $G'$ is not conjugate to $S_3$ or $A_4$. Also, if $G'$ is conjugate to $S_4$ then $Aut(C)$ is conjugate to $SmallGroup(120,5)$ or $SmallGroup(120,35)$ because these are the only groups of order $120$ with elements of order $10$ and no higher orders appear. But one can verify that there are no elements $\tau\in Aut(C)$ of order $3$ or $10$ that commute with $\sigma^5$ therefore $Aut(C)$ is not conjugate to any of these two groups. In particular, $G'$ is not conjugate to $S_4$.
On the other hand, groups of order $20$ that contain elements of
order $10$ and no higher orders are $SmallGroup(20,\ell)$ where
$\ell=1,3,4$ or $5$. There is no element $\tau\in Aut(C)$ of order
$4$ such that $\sigma^2\tau\sigma^2=\tau$ or
$\sigma^2\tau\sigma^{-2}=\tau\sigma^2$, which implies that
$\ell\neq1,3$. Furthermore, there is no element $\tau$ of order 2 in
$Aut(C)$ which commutes with $\sigma^5$ thus $\ell\neq4,5$. This
also implies that $G'$ is not conjugate to $C_4$. Lastly, groups of
order $300$ that contain elements of order $10$ and no higher orders
are $SmallGroup(300,\ell)$ where $\ell=25,26,27,41$ or $43$. If
$\ell=43$ or $41$ then $Aut(C)$ contains exactly $3$ element of
order $2$ which contradicts the fact that $Aut(C)$ should have at
least $15$ such elements as $A_5$ does. Moreover, there are no
elements of order 2 in $Aut(C)$ such that
$\tau\sigma^5=\sigma^5\tau$ hence $\ell\neq25,26$ or $27$.
Consequently, $G'$ is not conjugate to $A_5$. This proves the claim
in this situation.
\par
Secondly, assume that $C$ is a descendant of the degree 5 Fermat
curve. This should happen through a transformation
 $P\in PGL_{3}(K)$ such that $P^{-1}\sigma P=\lambda\sigma'_i$ where
 $\sigma'_1:=[X;\zeta_{10}^{2b}Z;\zeta_{10}^{2a}Y],\,\sigma'_2:=[\zeta_{10}^{2a}Y;X;\zeta_{10}^{2b}Z]$ and  $\sigma'_3:=[\zeta_{10}^{2b}Z:\zeta_{10}^{2a}Y;X]$ with $5\nmid(a+b)$.
 In each case, we get a Fermat descendant of the form
$$C_{P}:\,X^5+Y^5+Z^5+u\left(\xi_{10}^{a'}A^4B+AB^4\right)+u'\left(\xi_{10}^{b'}A^3B^2+A^2B^3\right)=0,$$
where $\{A,B\}\subset\{X,Y,Z\}$. Moreover, $C_{P}$ is the Fermat
curve only if $\beta_{2,0}^2=20$, and $Aut(C_{P})$ is cyclic of
order $10$ otherwise. For example, if ${P^{-1}\sigma
P}=\lambda\sigma'_1$ then $\lambda=\zeta_{10}^{2},\,\,5|a+b+2$, and
${P=[\zeta_{10}^{2a+2}\alpha_3Y+\alpha_3Z;X;-\zeta_{10}^{2a+2}\gamma_3Y+\gamma_3Z]}$.
Therefore $C$ is transformed into $C_{P}$ of the form
$$
X^5+Y^5+Z^5+u\left(\xi_{10}^{6(a+1)}Y^4Z+YZ^4\right)+u'\left(\xi_{10}^{2(a+1)}Y^3Z^2+Y^2Z^3\right)=0
$$
such that $\alpha _3
\left(\alpha_3^4+\beta_{2,0}\gamma_3^2\alpha_3^2+\gamma_3^4\right)=1$.
Now, $C_{P}$ is the Fermat quintic curve only if $u=0=u'$ or
equivalently $$5 \alpha _3^4+\beta _{2,0} \gamma _3^2 \alpha _3^2-3
   \gamma _3^4=0=5 \alpha _3^4-\beta _{2,0} \gamma _3^2 \alpha
   _3^2+\gamma _3^4.$$ Consequently, $\beta_{2,0}^2=20$.
(For instance when $K=\mathbb{C}$, one can take
$\alpha_3=-\frac{(-1)^{3/5}}{2^{4/5}}$, and
$\gamma_3=\frac{(-1)^{3/5}\sqrt[4]{5}}{2^{4/5}}$
 for $\beta_{2,0}=2\sqrt{5}$. Also, for $\beta_{2,0}=-2\sqrt{5}$, one
can assume $\alpha_3=\frac{1}{2^{4/5}}$, and $\gamma_3=-\frac{i
\sqrt[4]{5}}{2^{4/5}}$). Otherwise (i.e. $u\neq0$ or $u'\neq0$), one
could take $a=0$ and $b=3$ because all solutions (recall that
$[X;\zeta_{10}^{2b}Z;\zeta_{10}^{2a}Y]\in Aut(C_{P})$) are
$K$-isomorphic through a change of variables of the form $X\mapsto
X',\,Y\mapsto \xi_{10}^{\ell}Y'$ and $Z\mapsto \xi_{10}^{\ell'}Z'$.
Hence $C_{P}$ admits no more automorphisms inside $Aut(F_5)$, that
is, $Aut(C_{P})$, as a subgroup of $Aut(F_5)$, is cyclic of order
$10$.

For the other situations (i.e. by replacing $\sigma_1'$ with
$\sigma_2'$ and $\sigma_3'$), one can reduce to some concrete
$(a,b)$ and obtain exactly the same system to solve involving
$\beta_{2,0}$ as before. Thus the same conclusion follows.

\vspace*{-0.7cm}\[\qedhere\]
\end{proof}
\begin{rem}\label{rem1} Recall that $Aut(F_5)$ is generated by $\eta_1:=[X;Z;Y],\eta_2:=[Y;Z;X],\eta_3:=[\xi_5X;Y;Z]$ and $\eta_4:=[X;\xi_5Y;Z]$
 of orders $2,3,5$ and $5$ respectively such that $$(\eta_1\eta_2)^2=(\eta_1\eta_3)(\eta_3\eta_1)^{-1}=(\eta_3\eta_4)(\eta_4\eta_3)^{-1}=\eta_1\eta_4^2\eta_1(\eta_3\eta_4)^{-3}=\eta_2\eta_3\eta_2^{-1}(\eta_3\eta_4)^{-4}=1.$$
\end{rem}

\par The following lemma is very useful to discard all the groups with a
subgroup isomorphic to $C_2\times C_2$ for non-singular plane curves of
degree 5.
\begin{lem} \label{noafive} There is no non-singular plane curve $C$ of degree 5 with
$C_2\times C_2\preceq Aut(C)$. In particular, the full automorphism group $Aut(C)$ is not isomorphic to any of the groups: $C_2\times
C_2$, $A_4$, $S_4$ or $A_5$.
\end{lem}
\begin{proof}
By Mitchell \cite{Mit} and Harui \cite{Harui}, the group $C_2\times
C_2$ inside $PGL_3(K)$, which gives invariant a non-singular plane
curve $C$ of degree $d$ should fix a point not lying on $C$ or $C$
is a descendant of either the Fermat or the Klein curve. For $d=5$,
it could not be a descendant of the Fermat or the Klein curve
because $4$ does not divide $|Aut(F_5)|=150$ or $|Aut(K_5)|=39$.
Therefore, the automorphism subgroup $C_2\times C_2$ fixes a point
not in $C$. Moreover, because $2$ does not divide the degree $d=5$,
then by Harui's main theorem \cite{Harui}, we can think about the
elements of $C_2\times C_2$ in a short exact sequence: $1\rightarrow
N=1\rightarrow H\rightarrow H\rightarrow 1$, where $H$ is conjugate
to $C_2\times C_2$ inside $PGL_2(K)$. We can assume that $H$ acts
only on the variables $Y,Z$ because $N$ is the subgroup of $Aut(C)$
that acts on $X$. Now, let $\sigma,\tau\in H\subseteq PGL_2(K)$ be
of order two such that $\sigma\tau=\tau\sigma$ then, we can suppose,
up to a coordinate change inside $\mathbb{P}^2$, that
$\sigma=diag(1,-1)$ and $\tau=[aY+bZ,cY-aZ]\neq\sigma$.
Consequently, the curve $C$ has a model of type $2, (0,1)$. But all
possible $\tau$ does not retain invariant the equation of the type
$2, (0,1)$ for any choice of the free parameters and hence the
result follows. Indeed, because $\tau$ commutes with $\sigma$ then
$\tau=diag(-1,1)$ or $[X,bZ,cY]$ with $bc\neq 0$ and therefore $C$
has simultaneously the expressions: $Z^4 L_{1,Z}+Z^2
L_{3,Z}+L_{5,Z}$ and $Y^4L_{1,Y}+Y^2L_{3,Y}+L_{5,Y}$. Thus, $C$ has
the form $X\cdot G(X;Y;Z)$ with $G$ of degree 4 if $L_{5,Z}$ and
$L_{5,Y}$ are non-zero, a contradiction to irreducibility.
\end{proof}

Now, we deal with quintic curves whose automorphism group admits a
cyclic element of order 5 (respectively 4) as an automorphism of
maximal order.

\begin{prop}\label{type 5}
If $C$ is a degree 5, non-singular plane curve with an automorphism
$\sigma$ of maximal order 5, then we reduce, up to projective
equivalence, to one of the following situations: $Aut(C)$ is cyclic
of order 5 and $C$ is $K$-equivalent to the type $5, (0,1)$ of the
form $Z^5+L_{5,Z}$ or $Aut(C)$ is isomorphic to the Dihedral group
$D_{10}$ of order 10 where $Aut(C)=<\sigma,\tau>$ with
$\sigma(x,y,z)=(x,\xi_5 y,\xi_5^2 z)$ and $\tau(x,y,z)=(z,y,x)$, and
the curve $C$ has the form
$X^5+Y^5+Z^5+\beta_{3,1}X^2YZ^2+\beta_{4,3}XY^3Z=0$ such that
$(\beta_{3,1},\beta_{4,3})\neq(0,0)$.

\end{prop}

\begin{proof}
We consider the situations in \S1 concerning types $5, (a,b)$.
\begin{enumerate}
  \item Type $5, (1,2):$ $Aut(C)$ is not conjugate to any of the Hessian groups $Hess_{36}$
or $Hess_{72}$ and is not conjugate to a subgroup of $Aut(K_5)$,
since there are no elements of order $5$. On the other hand, always
$C$ admits a bigger automorphism group isomorphic to $D_{10}$
through the transformation $[Z;Y,X]$ (in particular, $Aut(C)$ is not
cyclic). Moreover, by the previous Lemma
\ref{noafive}, $Aut(C)$ is not conjugate to $A_5$ (as a finite primitive subgroup of $PGL_3(K)$). Consequently, $C$ is a descendant of the Fermat quintic or $Aut(C)$ fixes a line and a point off that line.
\begin{itemize}
\item If $Aut(C)$ fixes a line and a point off that line then this line should be $Y=0$ and the point is $[0:1:0]$ because
$<\sigma,\tau>\subseteq Aut(C)$ with
$\sigma(x,y,z)=(x,\xi_5y,\xi_5^2z)$ and $\tau(x,y,z)=(z,y,x)$. In
particular, elements of $Aut(C)$ are of the form
$[\alpha_1X+\alpha_3Z;Y;\gamma_1X+\gamma_3Z]$. Hence, from the
coefficients of $Y^3Z^2$ and $Y^3X^2$ (respectively, $X^4Y$ and
$YZ^4$), we should have $\alpha_1=0=\gamma_3$ or
$\alpha_3=0=\gamma_1$. Moreover $\alpha_{3,1}^5=\gamma_{1,3}^5=1$
and $\alpha_{3,1}\gamma_{1,3}=1$ or $(\alpha_{3,1}\gamma_{1,3})^2=1$,
since $(\beta_{3,1},\beta_{4,3})\neq(0,0)$. This implies that
$Aut(C)$ has order 10.

\item If $C$ is a descendant of the Fermat curve $F_5$ through a transformation $P\in PGL_3(K)$ and neither a line nor a point is leaved invariant.
Then, $P^{-1}[X;\xi_5Y;\xi_5^2Z]P=[X;\xi_5Y;\xi_5^2Z]$, because
elements of order $5$ in $Aut(F_5)$ are of the form
$\sigma_{a,b}:=[X;\xi_5^aY;\xi_5^bZ]$, and if
$P^{-1}[X;\xi_5Y;\xi_5^2Z]P=\lambda\sigma_{a,b}$ then
$(a,b)\in\{(1,2),(2,1),(3,4),(4,3),(1,4),(4,1)\}$, but all are
conjugate in $Aut(F_5)$. Now, $P$ has one of the forms $[X;\beta
Y;\gamma Z], [Y;\alpha Z;\beta X]$ or $[Z;\alpha X;\beta Y]$ and it
is straightforward to verify that there are no more automorphisms in
$Aut(F_5)\cap Aut(C_{P})$. That is, $Aut(C)$ is conjugate to
$D_{10}$.
\end{itemize}

  \item Type $5, (0,1):$ This curve has a homology $\sigma$ of order $d$ with center $[0:0:1]$ and axis $Z=0$ then (by Lemma $3.7$ in \cite{Harui}),
   this point is an outer Galois point of $C$. Moreover, $C$ is not $K$-isomorphic to the Fermat curve because automorphisms of $C$ has orders $\leq5$.
    Then (by Yoshihara \cite{Yoshihara}) this Galois point is unique and hence should be fixed by $Aut(C)$.
    In particular, $Aut(C)$ fixes a line ($Z=0$) and a point off that line ($[0:0:1]$). Thus elements of $Aut(C)$ have the form $[\alpha_1X+\alpha_2Y;\beta_1X+\beta_2Y;Z]$. Furthermore, $Aut(C)$ satisfies a short exact sequence  $1\rightarrow N\rightarrow Aut(C)\rightarrow G'\rightarrow 1$
with $N$ is cyclic of order dividing $5$ and $G'$ is conjugate to $C_m, D_{2m}, A_4, S_4$ or $A_5$ where $m\leq4$ and for the case  $G'=D_{2m}$ we have $m|3$ or $N$ is trivial.
\par If $N$ is trivial then $G'$ should be conjugate to $A_5$ (because none of the other groups contains elements of order 5) then $C_2\times C_2$ is a subgroup of $Aut(C)$ which is not possible by Lemma \ref{noafive}. Hence, $N$ can not be the trivial group.
\par If $N$ has order $5$ then for any value of $G'$ (except possibly the trivial group, $C_2$, $C_4$ or $A_4$ such that $Aut(C)$ is conjugate to $D_{10}$, $SmallGroup(20,3)$ or $A_5$) there are elements of order $>5$ in $Aut(C)$ a contradiction. Again, by Lemma \ref{noafive}, we conclude that $G'$ can not be $A_4$. On the other hand, there exists no elements $\tau\in Aut(C)$ of order $2$ such that $\tau\sigma\tau=\sigma^{-1}$ hence $G'$ is not $C_2$. Moreover, there are no elements $\tau\in Aut(C)$ of order $4$ such that $(\tau\sigma)^2=1$ and $\sigma\tau\sigma^{-1}=\tau\sigma$ thus $G^{'}$ is not conjugate to $C_4$. Consequently, $Aut(C)$ is cyclic of order 5.
\vspace*{-0.7cm}\[\qedhere\]
\end{enumerate}
\end{proof}

\begin{prop} Suppose that $C$ is a non-singular plane curve of degree 5 with $\sigma\in Aut(C)$
of order 4 as an element of maximal order, then we reduce, up to
$K$-isomorphism, to one of the following two situations: $Aut(C)$ is
cyclic of order $4$ which is generated by $\sigma(x,y,z)=(x, y,\xi_4
z)$, and $C$ is given by $Z^4Y+ L_{5,Z}(X,Y)=0$ such that
$L_{5,Z}(X,\zeta_mY)\neq\zeta^r_m
                L_{5,Z}(X,Y)$ where $(m,r)\in\{(8,1),(16,1),(20,4)\}$ or $\sigma(x,y,z)=(x,\xi_4 y,\xi_4^2 z)$
and $C$ is defined by
$X^5+X(Z^4+\alpha
Y^4)+\beta_{2,0}X^3Z^2)+\beta_{3,2}X^2Y^2Z+\beta_{5,2}Y^2Z^3=0$
such that $\alpha\beta_{5,2}\neq0$.
\end{prop}

\begin{proof}
We consider the situations in \S1 concerning types $4, (a,b)$.

First, we observe that $C$ can not be a descendant of the Fermat
curve $F_5$ or the Klein curve $K_5$ because $|Aut(F_5)|=150$ and
$|Aut(K_5)|=39$ and $4\nmid|Aut(F_5)|$ or $|Aut(K_5)|$, and $Aut(C)$
is not conjugate to $A_5$ since there are no elements of order 4.
Consequently, $Aut(C)$ is conjugate to $Hess_{36}, Hess_{72}$ or it
should fix a line and a point off that line by the result of Harui.
Moreover, for the last case, we need to consider the situation of a
short exact sequence of the form $1\rightarrow N=1\rightarrow
Aut(C)\rightarrow G'\rightarrow 1,$ where $G'$ should contain an
element of order $4$. That is, $G'$ is conjugate to a cyclic group
$C_4$ or a Dihedral group $D_{8}$ (by use of Lemma \ref{noafive}).
\begin{enumerate}
        \item Type $4, (1,3)$; All such curves
decompose into $X \cdot G(X,Y,Z)$ and therefore are reducible.
\item Type $4, (0,1)$; This curve admits a homology of order $d-1$ with center $[0:0:1]$ then it follows, by Harui
\cite{Harui}, that this point is an inner Galois point of $C$, and
moreover it is unique by Yoshihara \cite{Yoshihara}. Therefore, this
point should be fixed by $Aut(C)$ consequently, $Aut(C)$ is cyclic.
It also follows by our assumption that
$C$ is not conjugate to any of the above
that $Aut(C)$ is cyclic of order $4$. More precisely, we can rewrite type $4,(0,1)$ as
                $Z^4Y'+L_{5,Z}'(X,Y')=0$ and it is necessary to
                impose the condition that $L_{5,Z}'(X,\zeta_mY')\neq\zeta^r_m
                L_{5,Z}'(X,Y')$ where $(m,r)\in\{(8,1),(16,1),(20,4)\}$ (otherwise; we get a bigger automorphism group
                conjugate to those for types $8,(1,4),\,\,16, (1,12)$ or $20, (4,5)$).
            \item Type $4, (1,2)$; First, by the same reason as type $4,(1,3)$, we need to assume that $\beta_{5,2}\neq 0$.
            Secondly, we'll show that $Aut(C)$ is not conjugate to any of the Hessian subgroups $Hess_{36}$ or $Hess_{72}$ as follows:
            Both groups contain reflections but no four groups hence all reflections in the group will be conjugate to $[Z;Y;X]$
            (see Theorem $11$ in \cite{Mit}).
            Therefore, we can take $P\in PGL_3(K)$ such that $P^{-1}\,\sigma^2\,P=\lambda [Z;Y;X]$ and $Aut(C_{P})\subset PGL_3(K)$ is given by the
            usual presentation inside $PGL_3(K)$ of the above Hessian groups, in particular
            always $Aut(C_{P})$
            have the following five elements of $Hess_{36}$ and $Hess_{72}$: $[Z;Y;X]$, $[X;Z;Y]$, $[Y;X;Z]$, $[Y;Z;X]$ and $[X;\omega Y;\omega^2 Z]$,
             where $\omega$ is a primitive
            3rd root of unity.
Because $C_{P}$ is invariant through $[Z;Y;X]$, $[X;Z;Y]$, $[Y;X;Z]$
and $[Y;Z;X]$, then $C_{P}$ should be of the form:
$u(X^5+Y^5+Z^5)+a(X^4Z+X^4Y+Y^4X+Y^4Z+Z^4X+Z^4Y)+H(X;Y;Z)$ , where
$u,a\in K$ and $H(X,Y,Z)$ is a homogenous polynomial of degree 5
such that the degree of any of the variables is at most three. Now,
impose that $C_{P}$ is fixed by $[X,\omega Y,\omega^2 Z]$, we obtain
$u=0$ and $a=0$ a contradiction, because $C_{P}$ is non-singular.
Therefore there is no degree 5 curve with Hessian group $Hess_{36}$
or $Hess_{72}$.

Consequently, the claim follows and $Aut(C)$ should fix a line and a point off that line.
             \par Now, if $C$ admits a bigger non-cyclic automorphism group then it should be non-commutative by Harui and contain an element
            of order 2 which does not commute with $\sigma$. We can reduce to a subgroup which is conjugate to the dihedral
            group $D_8$, and moreover $Aut(C)$ fixes the point $[1:0:0]$ and the line $X=0$. That is, automorphisms of $C$ are of the form
            $[X;vY+wZ;sY+tZ]$. There is no element $\tau\in Aut(C)$ of order 2 such that $\tau\sigma\tau=\sigma^{-1}$. Hence $Aut(C)$ is not conjugate to $D_8$ or $S_4$. In particular, it is cyclic of order $4$. 
             \end{enumerate}\vspace*{-0.7cm}\[\qedhere\]
\end{proof}

Remains yet the study of curves $C$ where their automorphisms have
orders at most 3. In particular, $Aut(C)$ is not conjugate to $A_5,
Hess_{36}$ or $Hess_{72}$, because each of these groups contains
elements of order $>3$. Therefore, $Aut(C)$ should fix a line and a
point off that line or it is conjugate to a subgroup of $Aut(F_5)$
or $Aut(K_5)$.

\begin{prop} Let $C$ be a non-singular plane curve of type $3, (1,2)$ such that elements inside $Aut(C)$ have orders $\leq3$. Then, $C$ is $K$-isomorphic to $X^5+Y^4Z+YZ^4+\beta_{2,1}X^3YZ+X^2(\beta_{3,0}Z^3+\beta_{3,3}Y^3)+\beta_{4,2}XY^2Z^2=0$.
Moreover, if $\beta_{3,0}\neq\beta_{3,3}$ then $Aut(C)$ is cyclic of
order 3 and is conjugate to $S_3$ otherwise.
\end{prop}

\begin{proof}
If $C$ is a descendant of the Klein curve then $Aut(C)$ is conjugate to a subgroup of $Aut(K_5)$. Hence can not be of order $>3$, since $|Aut(K_5)|=3\cdot13$ (otherwise; $Aut(C)$ should contain an element of order $13>3$ by Sylow's theorem). 
\par If $C$ is a descendant of the Fermat curve then $Aut(C)$ is cyclic of order $3$ or conjugate to $S_3$ inside $Aut(F_5)$. Indeed, $|Aut(F_5)|=2\cdot3\cdot5^2$ hence any subgroup of order $>3$ is conjugate to $S_3$ (note that $Aut(F_5)$ contains no elements of order $6$) or it contains elements of order $5>3$. Now, if $Aut(C)$ is conjugate to $S_3$ then there exists $\tau\in Aut(C)$ of order $2$ such that $\tau\sigma\tau=\sigma^{-1}$ which reduces $\tau$ to be of the form $[X;\beta Z;\beta^{-1}Y]$. But, $[X;\beta Z;\beta^{-1}Y]\in Aut(C)$ iff $\beta^3=1$ and $\beta_{3,0}=\beta_{3,3}.$
\par If $Aut(C)$ fixes a point then should be one of the reference points $P_1:=[1:0:0],\,\,P_2:=[0:1:0]$ or $P_3:=[0:0:1]$,
 since these are the only points which are fixed by $\sigma$. If the fixed points is $P_2$ or $P_3$ then $Aut(C)$ is cyclic of order $3$
 because both points lie on $C$. If the fixed point is $P_1$ then the line that is leaved invariant should be $X=0$, hence automorphisms of $C$ have the form $[X;\beta_2Y+\beta_3Z;\gamma_2Y+\gamma_3Z]$. Moreover, it follows by Theorem \ref{teoHarui}(2), Lemma \ref{noafive} and the assumption that there are no elements in $Aut(C)$ of order $>3$ that $Aut(C)$ satisfies a short exact sequence of the form $1\rightarrow N=1\rightarrow Aut(C)\rightarrow G'\rightarrow 1,$ where $G'$ is conjugate to $C_3, S_3$.
This completes the proof.
\end{proof}

\begin{prop} Let $C$ be a non-singular plane curve of type $2, (0,1)$ such that elements inside $Aut(C)$ have orders $\leq2$. Then, $C$ is $K$-isomorphic to $C:Z^4L_{1,Z}+Z^2L_{3,Z}+L_{5,Z}=0$ and $Aut(C)$ is cyclic of order 2.
\end{prop}

\begin{proof}
$C$ is not a descendant of the Klein curve because $2\nmid
|Aut(K_5)|(=39)$. Also, if $C$ is a descendant of the Fermat curve
then $Aut(C)$ can not be conjugate to a bigger subgroup of
$Aut(F_5)$, because $|Aut(F_5)|=2.3.5^2$, thus subgroups of order
$>2$ should contain an element of order $3$ or $5$ which is a
contradiction. Finally, if $Aut(C)$ fixes a line and a point off
that line then, by \cite{Harui} and our assumption that there are no
automorphisms of order $>2$, we get that
 $Aut(C)$ satisfies a short exact sequence of the form  $1\rightarrow N=1\rightarrow Aut(C)\rightarrow G'\rightarrow 1,$  where $G'$ contain an
  element of order $2$ and no higher orders. Thus $Aut(C)$ should be conjugate to $C_2$ or $C_2\times C_2$.
  Consequently, the result follows by Lemma \ref{noafive}.
\end{proof}
Lastly, we need to ensure the existence of a non-singular plane
curve $C$, via certain specializations of the parameters, for which
the maximal order of the elements in its full automorphism group is
exactly $m$ where $m\leq 5$.

This is a tedious computation because we do not know a priori the
dimension of the locus $\rho(\widetilde{M_6(G)})$, see for example
the situations with $m=4$ in \cite{BaBacyc1}. To treat the case
$m\neq 4$, we can apply similar arguments as the situation $m=4$,
which will not be reproduced here (nevertheless, we know all the
possible groups and the representations that could appear such that
$m$ divides their order). This in turns simplifies the computations,
in order to conclude,

\begin{lem} Take $F(X;Y;Z)$ as the equation of degree 5 associated to Type $m,(a,b)$ where $m\leq 5$ in \S2, table \ref{table:Cyclic Auto.}
with $m, (a,b)\neq 4, (1,3)$.
 Then, there exists
a non-singular plane curve $C$ obtained by a concrete specialization
of the parameters of the equation $F(X;Y;Z)$, such that all the
elements of $Aut(C)$ are of order $\leq m$. Moreover, for type
$3, (1,2)$, we have curves with this property on the elements of
$Aut(C)$ of order $\leq 3$ satisfying $\beta_{3,0}\neq \beta_{3,3}$
and also satisfying $\beta_{3,0}=\beta_{3,3}$.
\end{lem}

Lastly, we remark the following:
\par In the next table, we list the exact groups $G$
(some of them are given in the GAP notations) that appear as an
$Aut(\delta)$ of a non-singular plane curve $\delta$ of genus 6. The
second column corresponds to the injective representation $\rho$ of
the group $G$ inside $PGL_3(K)$. The third column corresponds to a
homogenous equation of degree 5 associated with certain parameters
such that for any $\delta\in \rho(M_6^{Pl}(G))$, a plane
non-singular model associated to $\delta$ can be obtained by a
specialization of the parameters and vice versa.
\par It remains to assign to each equation
the parameters' restrictions to ensure that the equation is
geometrically irreducible, non-singular and it does not have a
bigger automorphism group (so that, any
$\delta\in\widetilde{\rho(M_6^{Pl}(G))}$ corresponds to some
specialization of the parameters with respect to the restrictions,
and conversely, any specialization of the parameters with respect
the restrictions gives a non-singular plane model of some element in
$\widetilde{\rho(M_6^{Pl}(G))}$). We recall that, always $\alpha(\neq
0)$ can be transformed by a diagonal change of variables to 1.

\begin{center}
\begin{table}[!th]
  \renewcommand{\arraystretch}{1.3}
  \caption{Full Automorphism of quintics}\label{table:FullAuto.}
  \centering
\begin{tabular}{|c|c|c|}
  \hline
  $G$ &$\rho(G)$& $F_{\rho(G)}(X;Y;Z)$  \\\hline\hline
  $(150,5)$& $[\xi_{5}X;Y;Z],[X;\xi_{5}Y;Z]$ & $X^5+Y^5+Z^5$  \\
           & $[X;Z;Y],\,[Y;Z;X]$   & \\\hline
  $(39,1)$& $[X;\xi_{13} Y;\xi_{13}^{10} Z],[Y;Z;X]$& $X^4Y+Y^4Z+Z^4X$   \\\hline
  $(30,1)$ & $[X;\xi_{15} Y;\xi_{15}^{11}Z],[X;Z;Y]$ & $X^5+Y^4Z+YZ^4$
  \\\hline
    $\Z/{20}\Z$& $[X;\xi_{20}^4 Y;\xi_{20}^5 Z]$ & $X^5+Y^5+XZ^4$    \\\hline
  $\Z/{16}\Z$& $[X;\xi_{16} Y;\xi_{12}^5 Z]$& $X^5+Y^4Z+XZ^4$    \\\hline
$\Z/{10}\Z$& $[X;\xi_{10}^2Y;\xi_{10}^5Z]$& $X^5+Y^5+ XZ^4+\beta_{2,0}X^3Z^2$\\
& &$\beta_{2,0}\neq0$ and $\beta_{2,0}^2\neq20$
\\\hline
$D_{10}$& $[X;\xi_5 Y;\xi_5^2 Z],\, [Z,Y,X]$ & $X^5+Y^5+Z^5+\beta_{3,1}X^2YZ^2+\beta_{4,3}XY^3Z$   \\
& &$(\beta_{3,1},\beta_{4,3})\neq(0,0)$   \\\hline
$\Z/8\Z$&$[X;\xi_{8} Y;\xi_{8}^{4} Z]$& $X^5+Y^4Z+
XZ^4+\beta_{2,0}X^3Z^2$\,\,\,$(\beta_{2,0}\neq0,\pm 2)$
\\\hline
 $S_3$& $[X;\xi_3Y;\xi_3^2Z]$& $X^5+Y^4Z+YZ^4+\beta_{2,1}X^3YZ+\beta_{3,3} X^2\big(Z^3+Y^3\big)+$  \\
 & $[X;Z;Y]$  &$+\beta_{4,2}XY^2Z^2$ (not above) \\\hline

$\Z/5\Z$& $[X;Y;\xi_5Z]$& $Z^5+L_{5,Z}$ (not above)  \\\hline
$\Z/4\Z$&  $[X;\xi_4Y;\xi_4^2Z]$& $X^5+X\big(Z^4+\alpha Y^4\big)+\beta_{2,0}X^3Z^2+\beta_{3,2}X^2Y^2Z+\beta_{5,2}Y^2Z^3$    \\
& &$(\beta_{5,2}\neq0)$\,  (not above)    \\\hline $\Z/4\Z$&
$[X;Y;\xi_4Z]$& $Z^4L_{1,Z}+L_{5,Z}$\,\,(not above)      \\\hline
 $\Z/3\Z$& $[X;\xi_3Y;\xi_3^2Z]$& $X^5+Y^4Z+\alpha YZ^4+\beta_{2,1}X^3YZ+$  \\
&&
$+X^2\big(\beta_{3,0}Z^3+\beta_{3,3}Y^3\big)+\beta_{4,2}XY^2Z^2$\,\,
(not above)   \\\hline
 $\Z/2\Z$&  $[X;Y;\xi_2Z]$& $Z^4L_{1,Z}+Z^2L_{3,Z}+L_{5,Z}$\,\,(not above)     \\\hline

 \end{tabular}
\end{table}

\end{center}

\

\begin{rem}
It should be noted that it appears here a new phenomena that did not
occur for degree $d=4$. Henn in \cite{He} observed that for each
finite group $G$ that appeared as an automorphism group of a
non-singular plane curve $\delta\in M_3^{Pl}(G)$, there is an unique
equation $F_G(X,Y,Z)$ (endowed with a set of restrictions on the
parameters), up to change of variables, where the specializations of
such equation is a plane non-singular model for the elements of
$M_3^{Pl}(G)$ and vice versa. This is not the case any more for
degree $5$, since from the above table, with the group $\Z/4\Z$, we
obtain two $\rho$'s where their equations $F_{\rho}(X;Y;Z)$ are not
$K$-isomorphic, and corresponds to the disjoint decomposition of
$\widetilde{M_6^{Pl}(\Z/4\Z)}$ in terms of non-empty
$\rho(\widetilde{M_6^{Pl}(\Z/4\Z)})$. Similar situations happen for
higher degrees, for more details, we refer to \cite{BaBacyc1}.
\end{rem}

\end{document}